%% file: k_ellipse_journal.tex
\documentclass{amsart}

\usepackage{amsmath}
\usepackage{amsthm}
\usepackage{amsfonts}
\usepackage{amssymb}
\usepackage{stmaryrd}
\usepackage{hyperref}
\usepackage{mathtools}
\usepackage{graphicx}
\usepackage{tabu}
\usepackage{tikz}
\usepackage{pgfplots}
\usepackage{listings}
\usepackage{marginnote}

\newtheorem{thm}{Theorem}
\newtheorem{lem}[thm]{Lemma}
\newtheorem{cor}[thm]{Corollary}
\newtheorem{prop}[thm]{Proposition}
\newtheorem{dfn}[thm]{Definition}
\newtheorem{eg}[thm]{Example}
\newtheorem{question}[thm]{Question}

\renewcommand{\leq}{\leqslant}
\renewcommand{\geq}{\geqslant}
\renewcommand{\setminus}{\smallsetminus}
\DeclareMathOperator{\Gal}{Gal}
\DeclareMathOperator{\val}{val}
\DeclareMathOperator{\tr}{tr}

\DeclareMathOperator{\Spec}{Spec}
\DeclareMathOperator{\hmg}{hmg}

\DeclarePairedDelimiterX\ip[2]{\langle}{\rangle}{#1, #2}

\usepackage{footmisc}
\DefineFNsymbols{mySymbols}{{\ensuremath\dagger}{\ensuremath\ddagger}\S\P
   *{**}{\ensuremath{\dagger\dagger}}{\ensuremath{\ddagger\ddagger}}}
\setfnsymbol{mySymbols}

\newcommand{\R}{\mathbb{R}}
\allowdisplaybreaks
\begin{document}

\title{Singularities and Genus of the $k$-Ellipse}

\author{Yuhan Jiang}
\address{Department of Mathematics, University of California, Berkeley}
\email{michelle.jiang@berkeley.edu}

\author{Weiqiao Han}
\address{Department of Electrical Engineering and Computer Science, Massachusetts Institute of Technology}
\email{weiqiaoh@mit.edu}






\newcommand{\CP}{\mathbb{CP}}
\newcommand{\Z}{\mathbb{Z}}
\newcommand{\N}{\mathbb{N}}
\newcommand{\C}{\mathbb{C}}
\begin{abstract}
A $k$-ellipse is a plane curve consisting of all points whose distances from $k$ fixed foci sum to a constant. We determine the singularities and genus of its Zariski closure in the complex projective plane. The paper resolves an open problem stated by Nie, Parrilo and Sturmfels in 2008.
\end{abstract}

\maketitle

\input{k_ellipse_journal_sections/introduction.tex}
\input{k_ellipse_journal_sections/preliminaries.tex}

\input{k_ellipse_journal_sections/singularities_at_infinity.tex}
\input{k_ellipse_journal_sections/nodal_singularities.tex}
\input{k_ellipse_journal_sections/dual_curve.tex}
\input{k_ellipse_journal_sections/conclusion.tex}
\input{k_ellipse_journal_sections/acknowledgement.tex}
\bibliographystyle{amsplain}
\bibliography{reference}

\end{document}

%% file: k_ellipse_journal_sections/introduction.tex
\section{Introduction}
The \textit{$k$-ellipse} for $k\in \N$ with \textit{foci} $(u_1, v_1), (u_2, v_2), \dots, (u_k, v_k) \in \mathbb{R}^2$ and \textit{radius} $r$ is the following curve in the plane:
\begin{align}\label{eq:kellipse}
    \left\{ (x,y) \in \mathbb{R}^2: \sum_{i=1}^k \sqrt{(x-u_i)^2 + (y-v_i)^2} = r \right\},
\end{align}
that is, the set of points in $\R^2$ whose distances from $k$ foci sum up to $r$. For example, a 1-ellipse is a circle, and a 2-ellipse is an ellipse in the usual sense. The curve bounds a convex set that has a linear matrix inequality (LMI) representation \cite{nie2008}. A \textit{(projective) algebraic $k$-ellipse} is the Zariski closure of a $k$-ellipse in $\mathbb{CP}^2$. It is an algebraic curve whose defining polynomial has a determinantal representation. We work with \emph{generic} algebraic $k$-ellipses in this paper. By generic, we mean all special cases one can think of are excluded; for example, we assume the foci are not collinear. In the configuration space of algebraic $k$-ellipses, the set of generic $k$-ellipses is a nonempty Zariski open (hence dense) set.

The degree of a generic algebraic $k$-ellipse is $2^k$ when $k$ is odd and $2^k - \binom{k}{k/2}$ when $k$ is even. For example, the degrees of algebraic $k$-ellipses for $k = 1,\ldots, 6$ are $2,2,8,10,32,44$. By the degree-genus formula \cite[Theorem 7.37]{kirwan_1992}, we may expect the genus to be, for example, 21 for a 3-ellipse, if the $k$-ellipse were nonsingular. However, the genus of generic algebraic $k$-ellipses for $k = 1,\ldots, 6$ is known to be $0, 0, 3, 6, 25, 55$, which implies that the curve is highly singular. How to find the genus of a generic algebraic $k$-ellipse for a general $k$, or how to describe its singularities were unknown. The genus is also related to the degree of the dual curve \cite[Chapter 7]{wall_2004}, and the degree of the dual curve is the algebraic degree \cite{nie2010algebraic} of the semidefinite representation of the $k$-ellipse \cite{nie2008}.
So the following questions were raised in \cite[Section 5]{nie2008}. 
\begin{question}
Is there a formula for the genus of the algebraic $k$-ellipse?
\end{question}
\begin{question}
Is there a nice geometric characterization of all singular points of the algebraic $k$-ellipse?
\end{question}
\begin{question}
Is there a formula for the degree of the dual curve of the algebraic $k$-ellipse?
\end{question}

In this paper, we address all three questions above. We prove the following theorems.
\begin{thm}\label{thm:genus}
Let $g_k$ denote the genus of an algebraic $k$-ellipse.
Then
\[
g_k = \left\{\begin{matrix} (k-2) 2^{k-2} + 1 & k\text{ is odd and } k \geq 3, \\ (k-2) 2^{k-2} - \binom{k-1}{k/2} + 1 & k \text{ is even.} \end{matrix} \right.
\]
\end{thm}

We will prove this theorem in Section 4. The algebraic $k$-ellipse is parametrized by its foci and its radius. These parameters $u_1,v_1,\ldots,u_k,v_k,r$ define the \textit{configuration space} $\Spec \C[u_1,v_1,\ldots,u_k,v_k, r]$  of the algebraic $k$-ellipse.

\begin{thm}\label{thm:singularities}
An algebraic $k$-ellipse has all singularities nodal except for $P := [\pm \mathrm{i}: 1: 0]$ where the multiplicity is twice the number of local branches.
\end{thm}
\begin{thm}\label{thm:dual}
Let $d_k^\vee$ denote the degree of the dual curve of the algebraic $k$-ellipse. Then
\[
d_k^\vee = \left\{ \begin{matrix} (k+1)2^{k-1} & k \text{ is odd,} \\ (k+1)2^{k-1} - 2\binom{k}{k/2} & k \text{ is even.} \end{matrix} \right.
\]
\end{thm}

\begin{figure}
\includegraphics[scale=.47]{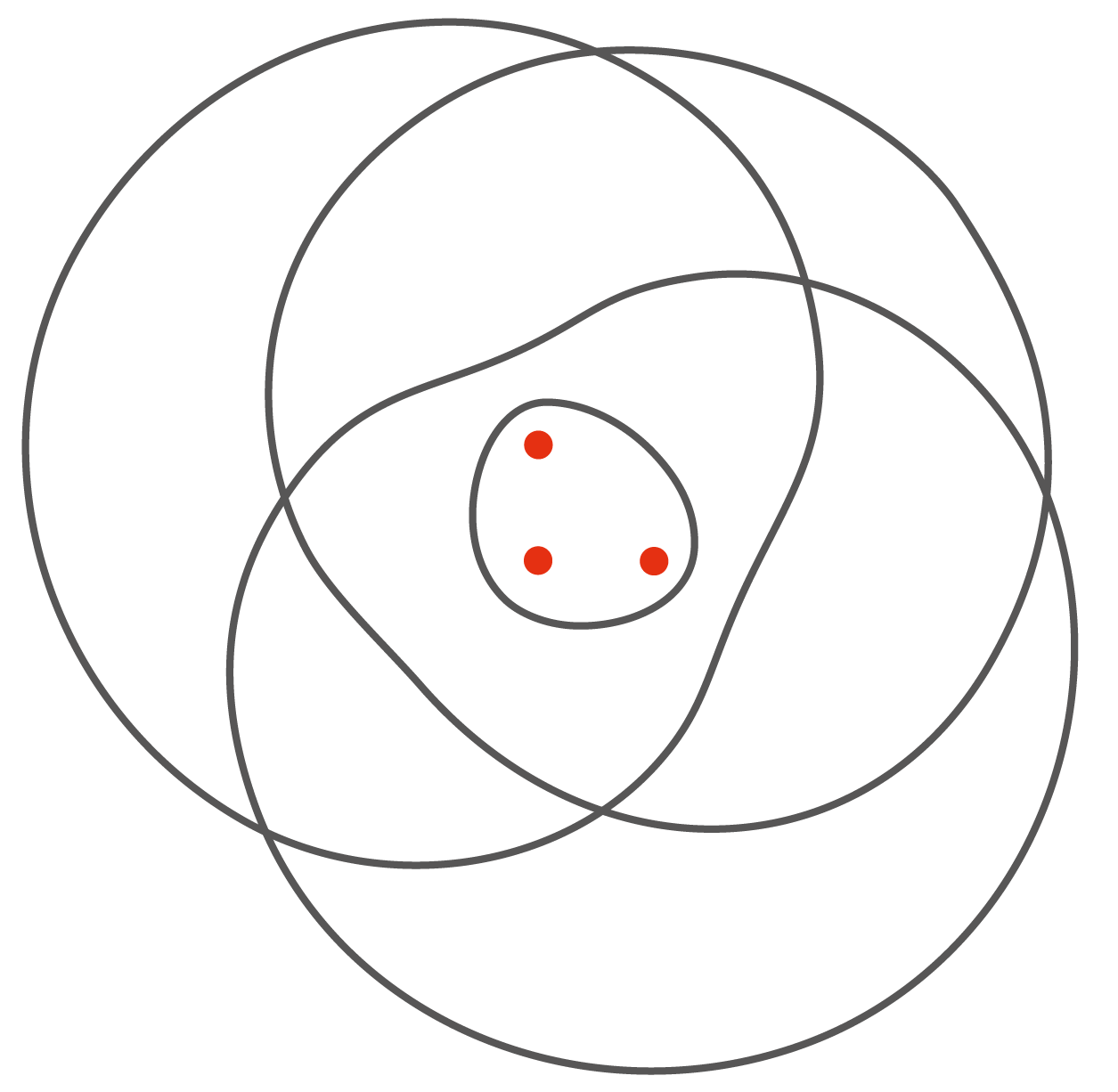}
\includegraphics[scale=.47]{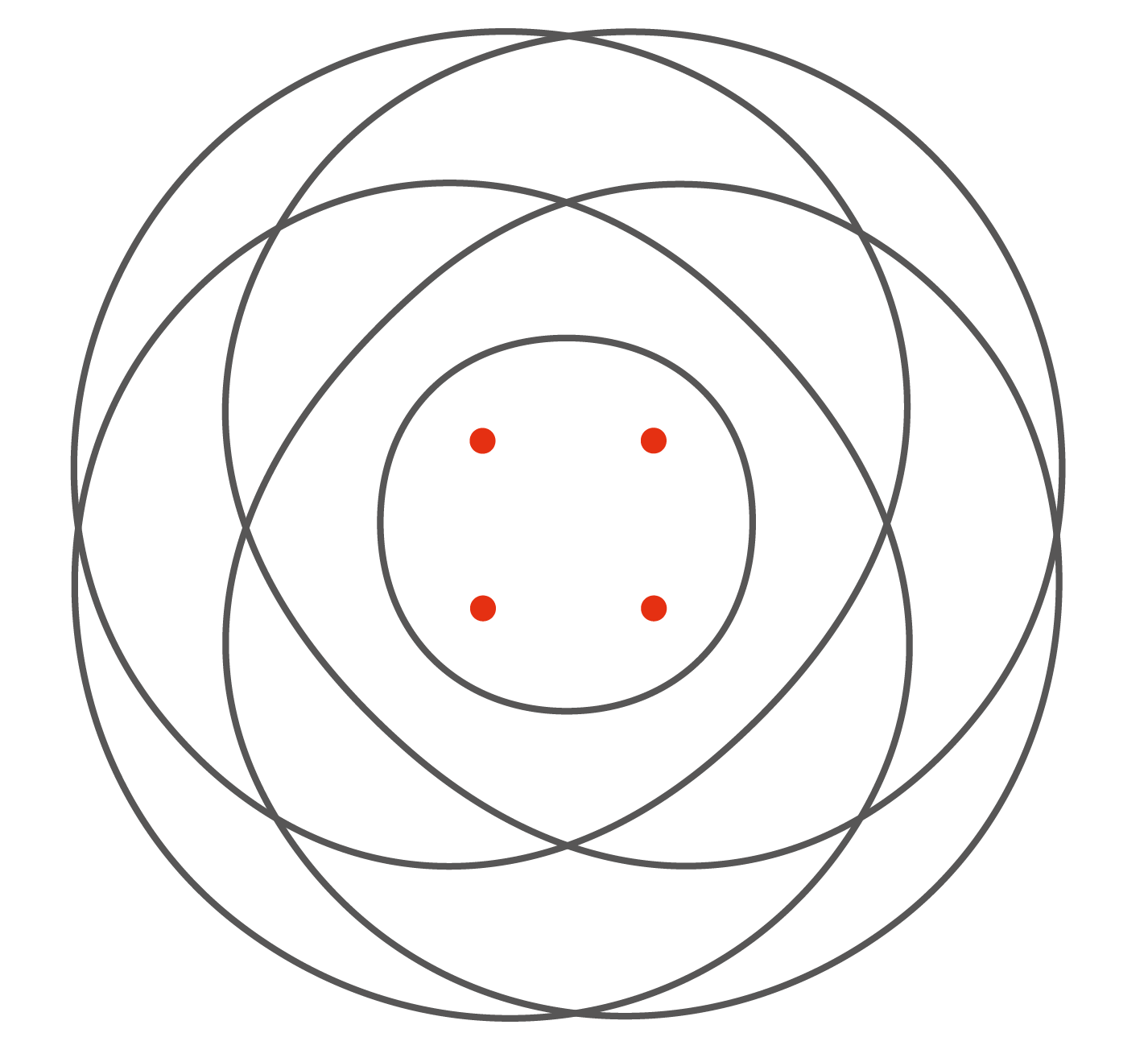}

\caption{Left: An algebraic 3-ellipse with foci (0,0), (1,0), (0,1) (in $U_z$) and radius 3. Right: An algebraic 4-ellipse with foci (0,0), (1,0), (0,1), (1,1) and radius 5.}
\end{figure}

%% file: k_ellipse_journal_sections/preliminaries.tex
\section{Singular Curves}\label{section:preliminaries}
In this section, we shall introduce terminologies and properties pertaining to singularities of an algebraic curve as well as basic properties of the $k$-ellipse.

Let $\mathbb{K}$ be a field. By $\mathbb{A}^n_{\mathbb{K}}$ we mean the affine $n$-space over $\mathbb{K}$. If $S$ is any set of polynomials in $\mathbb{K}[x_1,\ldots,x_n]$, we let $V(S) = \{P \in \mathbb{A}^n_{\mathbb{K}} | F(P) = 0 \text{ for all } F \in S\}$. If $S$ is a finite set $\{f_1,\ldots,f_s\}$, where $f \in \mathbb{K}[x_1,\ldots, x_n]$, we can also write $V(S)$ as $V(f_1,\ldots,f_s)$. When we work with affine plane curves, $n = 2$.

The complex projective plane $\CP^2$ is the quotient of the set $\{(x, y, z) \mid x,y,z \in \C \} \setminus \{(0,0,0)\}$ under the equivalence relation $\sim$ where for all $\lambda \in \mathbb{C} \setminus \{0\}, (x, y, z) \sim (\lambda x, \lambda y, \lambda z)$. Denote the equivalence class of $(x, y, z)$ by $[x:y:z]$. Let $U_z = \{[x: y: z]| z \neq 0, x, y, z \in \C\} = \{[x: y: 1]|x, y\in \C\} \cong \C^2 = \mathbb{A}_\C^2$. Define $U_x$ and $U_y$ similarly. Then $\{U_x, U_y, U_z\}$ is an open cover of $\CP^2$. An affine curve $f(x, y)$ is often defined on $U_z$, and its closure in $\CP^2$ is the projective curve defined by the homogenization $\hmg(f) = z^{\deg(f)} f(\frac{x}{z},\frac{y}{z})$ with an additional variable $z$. A point $[a : b: c]$ of a projective curve $C$ in $\CP^2$ defined by a homogeneous polynomial $f(x, y, z)$ is called \emph{a singularity} if $\frac{\partial f}{\partial x}(a, b, c) = \frac{\partial f}{\partial y}(a, b, c) = \frac{\partial f}{\partial z}(a, b, c) = 0$.

Let $C = V(f)$, where $f \in k[x,y]$ is squarefree, be an affine plane curve through the origin. We say that the curve $C$ is \emph{locally irreducible} in a neighborhood of the origin (0,0) if the function $f$ cannot be written as a product of two non-zero analytic functions $f(x,y) = f_1(x,y)f_2(x,y)$, both vanishing at the origin. If $f$ admits a factorization into locally irreducible factors $f = \prod_{i=1}^r f_i$, then the curves locally given by the $f_i$'s are called the \emph{local branches of the curve at the origin} \cite[Definition 5.1]{kaza2018}.
We call the smallest degree $m$ of a non-zero term of $f$ the \textit{multiplicity} of $C$ at the origin. Let $f_m$ denote the degree-$m$ part of the polynomial $f$. Then $V(f_m)$ is called the \textit{tangent cone} of $C$ at the origin. The tangent cone $V(f_m)$ contains tangent lines to the branches of $C$ at the origin. To find the tanget cone of any point $P = (a,b) \neq (0,0)$ on $f$, we can translate $P$ to the origin, or directly Taylor-expand $f$ at $P$.
For an algebraic curve $C \subseteq \mathbb{CP}^2$, the above terminologies are defined on $C \cap U_x$, $C \cap U_y$, or $C \cap U_z$.
For example, if a point $P \in C \cap U_x \subseteq \mathbb{CP}^2$ has multiplicity $m$ in $U_x$, then we say $P$ has multiplicity $m$. It is well-defined in the sense that if $P \in C \cap U_x \cap U_y$ has multiplicity $m_1$ in $U_x$ and multiplicity $m_2$ in $U_y$, then $m_1 = m_2$. Denote by $m_P$ the \emph{multiplicity} at $P$ and denote by $r_P$ \emph{the number of local branches} at $P$. 

Let $\tilde{p}_k(x,y)$ be the irreducible polynomial that vanishes on the $k$-ellipse with foci $(u_1,v_1),\ldots,(u_k,v_k)$ and radius $r$ defined by (\ref{eq:kellipse}).
Let $A = (a_{ij})$ be a real $m \times m$ matrix and $B$ a real $n \times n$ matrix. Recall the Kronecker product $A \otimes B$ is the $mn \times mn$ matrix of the block form $\begin{pmatrix} a_{11} B & \cdots & a_{1m} B \\
\vdots & \ddots & \vdots \\
a_{m1} B & \cdots & a_{mm} B \end{pmatrix}$ as a linear map on the tensor product $\R^m \otimes \R^n$ with a standard choice of basis. Define the \textit{tensor sum} of $A$ and $B$ as the $mn \times mn$ matrix $A \oplus B := A\otimes I_n + I_m \otimes B$.
Define the $2^k\times 2^k$ matrix 
\begin{align}
    L_k(x,y) = r \cdot I_{2^k}  + \left[\begin{matrix}x-u_1 & y-v_1\\ y-v_1 & -x+u_1 \end{matrix}\right] \oplus \cdots \oplus \left[\begin{matrix}x-u_k & y-v_k\\ y-v_k & -x+u_k \end{matrix}\right],
\end{align}
which is affine in $x, y$ and $r$.
By \cite[Theorem 2.1]{nie2008}, the $k$-ellipse admits the determinantal representation $\tilde{p}_k(x,y) = \det L_k(x,y)$.

We homogenize $\tilde{p}_k(x,y)$ by introducing an additional variable $z$, and consider the homogeneous polynomial $p_k(x,y,z)$ in the complex projective space $\CP^2$. 
Then $V(p_k)$ is the \textit{algebraic $k$-ellipse} and $V(p_k) \cap U_z$ is the (affine) $k$-ellipse defined by (\ref{eq:kellipse}).
Let
\begin{equation}\label{galois}
q_k(x, y, z) = \prod_{\sigma \in \{-1, +1\}^k} \left( rz - \sum_{i=1}^k \sigma_i \sqrt{(x - u_i z)^2 + (y - v_i z)^2} \right).
\end{equation}
By \cite[Theorem 1.1]{nie2008}, 
\begin{equation} \label{pk}
p_k = \left\{ \begin{matrix} q_k & k \text{ is odd} \\
z^{-\binom{k}{k/2}} q_k & k \text{ is even} \end{matrix} \right.
\end{equation}

\begin{eg}
A generic algebraic 3-ellipse with foci $(0,0), (1,0), (0,1)$ (in $U_z$) and radius $r$ has 8 singularities: $[\pm \mathrm{i}: 1: 0]$ with multiplicity 4, $[\pm {\sqrt{2r^2} \over 2}:\pm {\sqrt{2r^2} \over 2}:1]$ with multiplicity 2, $[{1\over 2}: 1\pm {1\over 2}\sqrt{-1+4r^2}: 1]$ with multiplicity 2, and $[ 1\pm {1\over 2}\sqrt{-1+4r^2}: {1\over 2}: 1]$ with multiplicity 2.
The number of local branches at every singularity equals to its multiplicity, except that
$[\pm \mathrm{i}: 1: 0]$ have the number of local branches half of their multiplicity.
\end{eg}

Let $C, L$ be curves in $\mathbb{CP}^2$ defined by ideals $I_C, I_L$. We define the \emph{intersection multiplicity} of $C$ and $L$ at the point $P$, denoted $I(P, C \cap L)$, as the length of the localized module $(\mathbb{C}[x,y,z]/(I_C + I_L))_P$. The multiplicity of $P$ on $C$ can likewise be defined as the length of the localized module $(C[x,y,z]/I_C)_P$, which agrees with the previous definition. Indeed, if $I_C$ is principally generated by the homogenization of $f$, $\hmg(f) \in \mathbb{C}[x,y,z]$, and if $P = [0:0:1]$ and $\hmg(f)(0,0,1) = 0$, then the length of $(C[x,y,z]/I_C)_P$ is the smallest degree of terms in $f(x,y)$.

\begin{thm}[B\'ezout's Theorem, \cite{fulton89}]
If $C$ and $D$ are two projective curves of degrees $n$ and $m$ in $\mathbb{CP}^2$ which have no common component then they have precisely $n m$ points of intersection counting multiplicities; i.e.
\[
\sum_{p \in C \cap D} I_p(C, D) = nm.
\]
\end{thm}

A point $P$ on an algebraic curve $C$ is called a \emph{node} or a \emph{nodal singularity} if $m_P=r_P=2$ and there are two distinct tangent directions at $P$. If $m_P = 2$ and $P$ has only one tangent $L$, $P$ is called a \emph{cusp} if $I(P, C \cap L) \geq 3$; an \emph{ordinary cusp} if $I(P, C \cap L) = 3$.

Let $d$ be the \emph{degree} of $C$ and $d^\vee$ be the degree of its \emph{dual curve} $C^\vee$ respectively. By degree of a curve we mean the degree of the its defining polynomial. 
The genus $g$ of a curve is related to it degree and singularities by Noether's formula:
\begin{thm}[\cite{kirwan_1992}, Theorem 7.37]\label{thm:genus_formula}
The genus of an irreducible projective curve $C$ of degree $d$ in $\CP^2$ is
\begin{align}
    g = \binom{d-1}{2} - \sum_{P \in Sing (C)} \delta_P,
\end{align}
where $Sing (C)$ is the set of all singularities of $C$, and $\delta_P$ is the delta invariant at $P$ that we shall define shortly.
\end{thm}
There are several equivalent definitions of delta invariant. Wall \cite{wall_2004} and Kirwan \cite{kirwan_1992} take the topological point of view, while Serre \cite{serre1988} defines it algebraically. Consider an irreducible algebraic curve $C$ with coordinate ring $\mathfrak{o}$. Let $\bar{\mathfrak{o}}$ be the integral closure of $\mathfrak{o}$ in its fraction field. A point $P$ on $C$ corresponds to a maximal ideal $\mathfrak{p} \subset \mathfrak{o}$. Let $\mathfrak{o_p}$ be $\mathfrak{o}$ localized at $\mathfrak{p}$. Let $\mathfrak{\overline{o_p}} := \bigcap_{\mathfrak{q} \cap \mathfrak{o} = \mathfrak{p}} \mathfrak{\bar{o}_q} \subset \bar{\mathfrak{o}}$. Following \cite[Chapter IV.1]{serre1988}, define the \emph{delta-invariant} at $P$ as $\delta_P := \dim_{\mathbb{C}}(\overline{\mathfrak{o_p}}/\mathfrak{o_p})$, the dimension of $\overline{\mathfrak{o_p}}/\mathfrak{o_p}$ as a finite dimensional $\mathbb{C}$-vector space. Note that $\delta_P$ is invariant under completion. Two singularities are said to be analytically isomorphic if they have the same complete local ring, and it follows that $\delta$ is an analytic invariant.

Another definition of $\delta$ depends on the notion of blow-up. 
Construct the blow-up of $\mathbb{C}^2$ at the point $(0,0)$ as follows \cite[Section 3.3]{wall_2004}. Consider the product $\mathbb{C}^2 \times \mathbb{P}^1$. Let $x, y$ be the affine coordinates of $\mathbb{C}^2$ and $\eta, \xi$ be the homogeneous coordinates of $\mathbb{P}^1$. Define the \emph{blow-up} of $\mathbb{C}^2$ at (0,0) to be the closed subset $X$ of $\mathbb{C}^2 \times \mathbb{P}^1$ defined by $x\eta = y\xi$. The projection from the product to $\mathbb{C}^2$ defines a map $\pi: X \to \mathbb{C}^2$. The preimage of (0,0) is the entire projective line. By \cite[Section 3.2]{wall_2004}, the $\mathbb{P}^1$ corresponding to the affine origin is called the \emph{exceptional curve} of the blow up. We say \emph{$Q$ is infinitely near $P$} if $Q$ lies on any exceptional curve obtained by blowing up $P$ successively until there are no singular points left.
\begin{thm}[\cite{wall_2004}, Theorem 6.5.9 and its Remark]\label{thm:delta_and_multiplicity}
For a point $P$ on a curve $C$ with multiplicity $m_P$, $\delta_P = \sum_P \binom{m_P}{2}$ summing over all infinitely near singular points lying over $P$ including $P$.
\end{thm}

The following examples illustrate the two definitions of delta invariant.

\begin{eg}
Consider an affine plane curve analytically irreducible at the origin, isomorphic to $y^n - x^m$. Then $\gcd(n, m) = 1$. Suppose $n < m$ and $m = qn + r$ in the Euclidean algorithm sense \cite[Example 3.6.2]{wall_2004}. Then $\frac{y^j}{x^{n\ell+i}}$ for $1 \leq \ell \leq q, 1 \leq i \leq j < n$ are integral over the local ring $\mathfrak{o}_{(0,0)}$. They correspond to the first $q$ blow-up where the origin is always infinitely near of multiplicity $n$. There are $q \binom{n}{2}$ of them. Again let $n = q_1 r + r_1$ in the Euclidean algorithm sense, then $\frac{x^{i+r\ell+qj}}{y^j}$ for $1 \leq \ell \leq q_1, 1 \leq j \leq i < r$ are also integral over the local ring. They correspond to the $q_1$ blow-up afterwards where the origin is always infinitely near of multiplicity $r$. There are $q_1 \binom{r}{2}$ of them. The process continues until Euclidean algorithm terminates, and the singularity is resolved.
\end{eg}
\begin{eg}
Consider an affine plane curve isomorphic to $x^m - y^m$ at the origin. Then $\frac{y^j}{x^i}$ for $1 \leq i \leq j < m$ are integral over the local ring. The first blow up by $(x,y) \mapsto (x_1, x_1y_1)$ has $x_1=0$ as exceptional curve and $y_1 = e^{2\pi \ell \mathrm{i}/m}$ for $\ell = 0,\dots, m$ as strict transforms.
\end{eg}

%% file: k_ellipse_journal_sections/singularities_at_infinity.tex
\section{Singularities at Infinity}
In this section, we study the singularities of the algebraic $k$-ellipse at infinity.
We locate the singularities (Proposition \ref{prop:sing_at_inf_locations}), compute their tangent cones and hence multiplicities (Proposition \ref{prop:tangentcone}), and examine their local branches, which leads to the relation $2r_P = m_P$ (Proposition \ref{prop:branches}).

Let $R = \mathbb{C}  [u_1, v_1,\dots, u_k, v_k, r, x, y, z ]$ and $K$ be its fraction field.

\begin{prop}\label{prop:sing_at_inf_locations}
An algebraic $k$-ellipse intersects the line $z=0$ at
\begin{itemize}
\item only $[\pm \mathrm{i}: 1: 0]$ of multiplicity $2^{k-1}$ when $k$ is odd;
\item $[\pm \mathrm{i}: 1: 0]$ of multiplicity $2^{k-1} - \binom{k}{k/2}$ and $\binom{k}{k/2}$ points in quadratic extensions of $K$ when $k$ is even. The quadratic extensions are given by $r^2(x^2 + y^2) - (x \sum \sigma_i u_i + y \sum \sigma_i v_i)^2$.
\end{itemize}
\end{prop}
\begin{proof}
Consider the field extension $L$ joining $\sqrt{(x- u_i z)^2 + (y- v_i z)^2}$ to $K$ for $i=1, \ldots, k$. Evaluation at $z=0$ gives a field homomorphism from $L$ to $K(\sqrt{x^2 + y^2})$. 

For any binary vector $\sigma \in \{1, -1\}^k$, denote the sum of all entries by $|\sigma| := \sum_{i=1}^k \sigma_i$. When $k$ is odd, $|\sigma| \neq 0$ for all $\sigma \in \{\pm 1\}^k$ so $\sum_{i=1}^k \sigma_i \sqrt{x^2 + y^2} = |\sigma| \sqrt{x^2+y^2} \neq 0$. We have $p_k(x, y, 0) = \prod_{\sigma \in \{-1, +1\}^k} \sum_{i=1}^k \sigma_i \sqrt{x^2 + y^2}$. Then $p_k(x, y, 0)$ is a constant multiple of $(x^2 + y^2)^{2^{k-1}}$. This proves the case when $k$ is odd.

For each $\sigma \in \{1, -1\}^k$, let 
\[
f_\sigma (x, y, z) := rz - \sum_{i=1}^k \sigma_i \sqrt{(x - u_i z)^2 + (y - v_i z)^2}.
\]
When $k$ is even, there are $\binom{k}{k/2}$ many $\sigma$'s with $|\sigma|=0$. By L'H\^opital's rule, $\frac{f_\sigma}{z}(x, y, 0) = \lim_{z \to 0} \frac{\partial f_\sigma(x, y, z)}{\partial z}$ if limit exists. The limit exists in the affine chart $x \neq 0$. For these $\sigma$'s, compute $\frac{\partial f_\sigma}{\partial z}$ in the fraction field of $L$ and evaluate at $z=0$, we have
\[
\frac{\partial f}{\partial z}(x,y,0) = \frac{r \sqrt{x^2+y^2} + x \sum \sigma_i u_i + y \sum \sigma_i v_i}{\sqrt{x^2+y^2}}.
\]
By \cite[Theorem 1.1]{nie2008}, $p_k = z^{-\binom{k}{k/2}}\prod_\sigma f_{\sigma}$,
so 
\[p_k(x, y, 0) = M (x^2+y^2)^{2^{k-1}-\binom{k}{k/2}} \prod_{|\sigma| = 0} \left( r \sqrt{x^2+y^2} + x \sum \sigma_i u_i + y \sum \sigma_i v_i \right).
\]
holds, where $M\neq 0$ is a constant.
\end{proof}

When $k$ is even, the $\binom{k}{k/2}$ points on the line $z=0$ in quadratic extensions of $K$ are generically nonsingular.

Next we study the multiplicity and the number of local branches of the singularities at infinity. For clarity we discuss all results for the point $[-\mathrm{i}: 1: 0]$, and results for $[\mathrm{i}: 1: 0]$ follow by symmetry. 
We compute the multiplicity of $[-\mathrm{i}: 1: 0]$ by translating $[-\mathrm{i}: 1: 0]$ to the origin $[0:1:0]$ and computing the tangent cone at the origin.
Define
\[
t_k(x, z) := \prod_{\sigma \in \{\pm1\}^k} \sum \sigma_j \sqrt{- \mathrm{i} x + (\mathrm{i} u_j - v_j) z}
\]
Then $t_k \in K$ since the product is over the orbit of $\sum \sigma_j \sqrt{- \mathrm{i} x + ( \mathrm{i} u_j - v_j) z}$ under $\Gal(L/K)$. Each factor in the product is integral over $R$ so $t_k \in R$. Due to the automorphism $\sigma \mapsto -\sigma$ which flips all digits, $t_k$ (up to sign) is a square.

\begin{prop}\label{prop:tangentcone}
Let $\mathbb{C} \{\{u_1,\dots,u_k, v_1,\dots,v_k, r, x, y, z\}\}$ be the field of Puiseux series. We consider the point $[-\mathrm{i}:1:0]$ of the algebraic $k$-ellipse in the open set $U_y$ and translate $[-\mathrm{i}:1:0]$ to the origin $[0:1:0]$. 
Then the tangent cone of the algebraic $k$-ellipse at the origin $[0 :1: 0]$ is $V(t_k)$ when $k$ is odd and $k \geq 3$; $V(z^{-\binom{k}{k/2}} t_k)$ when $k$ is even.
\end{prop}
\begin{proof}
We translate $[-\mathrm{i}:1:0]$ to the origin $[0:1:0]$ by making the substitution $x-\mathrm{i} \mapsto x, z \mapsto z$ in the defining polynomial $q_k(x, y, z)$ of the algebraic $k$-ellipse:
\[
q_k(x - \mathrm{i}, 1, z) = \prod_{\sigma \in \{-1, 1\}^k} \left( rz - \sum_{j=1}^k \sigma_j \sqrt{2(-\mathrm{i} x + (\mathrm{i} u_j - v_j) z) + x^2 - 2 u_j x z + (u_j^2 + v_j^2) z^2} \right).
\]
The initial form of $q_k(x - \mathrm{i}, 1, z)$, of degree $2^{k-1}$, is (a $2^{2^{k-1}}$ multiple of)
\[
t_k(x, z) = \prod_{\sigma \in \{-1, 1\}^k} \sum_{j=1}^k \sigma_j \sqrt{- \mathrm{i} x + (\mathrm{i} u_j - v_j) z}
\]
Therefore, the multiplicity of $[\mathrm{i}: 1: 0]$ is $2^{k-1}$.
The initial form $t_k$ is homogeneous in $x$ and $z$, so it factors into $2^{k-2}$ squares of linear forms.
Since we assume the algebraic $k$-ellipse is generic, there are $2^{k-2}$ distinct tangent lines at $[\mathrm{i}: 1: 0]$.

When $k$ is even, an algebraic $k$-ellipse is defined by $z^{-\binom{k}{k/2}} q_k$. Let $\val_z: P \to \mathbb{Z}$ be the valuation at the prime $z$. By the strong triangle inequality $\val_z(q_k(x-\mathrm{i},1,z)) \geq \val_z(q_k) = \binom{k}{k/2}$ so $z^{-\binom{k}{k/2}} t_k$ is a polynomial.
\end{proof}

Every local branch at a point $P$ admits an analytic parametrization with a unique tangent at $P$ \cite[Lemma 2.4.4]{wall_2004}. The proof of Proposition \ref{prop:tangentcone} shows that there are $2^{k-2} = \frac{m_P}{2}$ distinct tangent lines at $[-\mathrm{i}: 1: 0]$, so there are at least $\frac{m_P}{2}$ branches at $[-\mathrm{i}: 1: 0]$. We would like to show there are at most that many.
\begin{prop}\label{prop:branches}
The $k$-ellipse factors into a ``bouquet of cusps" at $P = [-\mathrm{i}: 1: 0]$. As a consequence, $r_P = \frac{m_P}{2}$.
\end{prop}
By \cite[Theorem 2.4.3]{wall_2004} the field of univariate Puiseux series is algebraically closed. Therefore $q(x-\mathrm{i}, 1, z) = 0$ admits solutions of the form $x = p(z)$ in $\mathbb{C}\{\{z\}\}$ or $z = p(x)$ in $\mathbb{C}\{\{x\}\}$. We analyze solutions that give a cusp.

Recall that for a point on an algebraic curve $C$, if $m_P = 2$ and $P$ has only one tangent $L$, $P$ is called a cusp if $I(P, C \cap L) \geq 3$; an ordinary cusp if $I(P, C \cap L) = 3$. In local coordinates $x, y$, an ordinary cusp admits two solutions in the field of Puiseux series of the form $y = x f(x) + x^{3/2} g(x) = x (f(x) + \sqrt{x}g(x)), y = x f(x) - x^{3/2} g(x) = x (f(x) - \sqrt{x}g(x))$ for some units $f, g$, i.e., polynomials with a non-zero constant term, in the ring of power series $\mathbb{C}\llbracket x \rrbracket$. The constant term of $f$ gives the tangent at $(0,0)$, so the two solutions share tangents.
\begin{proof}
We translate $P = [-\mathrm{i}: 1: 0]$ to the origin $[0:1:0]$ as before.
We are going to show that any tangent line $V(x-az)$ at the origin corresponds to a cusp, by showing that $x$ can be written in the form of $x = z(f(z) + \sqrt{z}g(z))$.
Since we have $\frac{m_P}{2}$ distinct tangents with multiplicity 2 each, if we found $m_P$ parametrizations where every distinct tangent appears as the tangent of two parametrizations, we have found all parametrizations and the proof will be complete.
We view $q_k(x-\mathrm{i}, 1, z)$ as a polynomial in $x$ with coefficients in the field of Puiseux series about $z$. 
Write $x = x_1 z$.
\[
q_k(x_1 z - \mathrm{i}, 1, z) = z^{2^{k-1}}\prod_{\sigma \in \{-1, 1\}^k} \left( r\sqrt{z} -  \varphi_\sigma(x_1,z) \right)
\]
where 
\[
\varphi_\sigma(x, z) = \sum_{j=1}^k \sigma_j \sqrt{2(- \mathrm{i} x + (\mathrm{i} u_j - v_j) ) + x^2 z - 2 u_j x z + (u_j^2 + v_j^2) z}.
\]
By writing $x_1 = x_2 + a$, we have
\[
q_k((x_2+a)z - \mathrm{i}, 1, z) = z^{2^{k-1}}\prod_{\sigma \in \{-1, 1\}^k} \left( r\sqrt{z} -  \varphi_\sigma(x_2+a, z) \right)
\]
Since $V(x-az)$ is a tangent line, $(x-az)$ is a factor of $t_k$, so $(a,1) \in V(t_k)$, i.e., $\prod_\sigma \sum \sigma_j \sqrt{(\mathrm{i} u_j - v_j) - \mathrm{i} a} = 0$. 
For some $\tilde{\sigma}$, $\sum \tilde{\sigma}_j \sqrt{(\mathrm{i} u_j - v_j) - \mathrm{i} a} = 0$.
This term is the constant term of the Taylor expansion of $\varphi_{\tilde{\sigma}} (x_2+a, z)$ at $(0,0)$, and hence the Taylor expansion of $\varphi_{\tilde{\sigma}}(x_2+a, z)$ at $(0,0)$ is of the form $\sum_{i+j \geq 1, i,j \in \N} a_{ij} x_2^i z^j$. 
We may assume generically that $a_{10} \neq 0$, so $x_2^2$ appears with nonzero coefficient. 
Thus, $(r\sqrt{z} -  \varphi_{\tilde{\sigma}}(x_2+a, z)) (r\sqrt{z} -  \varphi_{-\tilde{\sigma}}(x_2+a, z)) = r^2 z - \varphi_{\tilde{\sigma}}(x_2+a, z)^2$ has no constant term. 
By \cite[Theorem 2.1.1]{wall_2004}, $r^2 z - \varphi_{\tilde{\sigma}}(x_2+a, z)^2$ admits a solution of the form $z = t^n, x_2 = \sum_{r=1}^\infty \alpha_r t^r$, and \cite[Corollary 2.4.2]{wall_2004} ensures that all solutions assume this form.
Write $x_2 = b_m z^m + p(z)$, where $m = r/n$, and $p(z)$ represents higher order terms.
$r^2 z - \varphi_{\tilde{\sigma}}(x_2+a, z)^2 = r^2 z - \sum_{i+j \geq 2 \text{ and } i, j \in \N} c_{ij} x_2^i z^j = r^2z - \sum_{i+j \geq 2 \text{ and } i, j\in \N} (b_m z^m + p(z))^i z^j = r^2z - (\sum_{i+j \geq 2 \text{ and } i, j \in \N} d_{ij} z^{im+j}$ + higher order terms), for some coefficients $c_{ij}, d_{ij}$.
The minimum degree of $z$ is among the indices $\{i m+j \mid i+j\geq 2 \text{ and } i, j\in \N \} \cup \{1\}$. If $m < \frac{1}{2}$, the minimum degree term of $z$ appearing is $d_{20}z^{2m}$ which we have assumed to be nonzero, contradiction. If $m=\frac{1}{2}$, the minimum degree term of $z$ is $r^2z - d_{20}z$ so we set $d_{20} = r^2$. If $m > \frac{1}{2}$, the minimum degree term of $z$ is $r^2z$ where we have assumed $r > 0$, contradiction.

Thus we may choose a suitable branch of $\sqrt{z}$ such that $x = z(a+\sqrt{z}p(z))$ for some Puiseux series $p(z)$ parametrizes $q_k(x_2+a - \mathrm{i}, 1, z)$ at $(0,0)$.
\end{proof}

\begin{prop}\label{prop:sing_at_inf_no_infly_near_sing}
The point $P=[-\mathrm{i}: 1: 0]$ has no infinitely near singular point.
\end{prop}
\begin{proof}
An ordinary cusp with tangent $y=ax$ admits a local parametrization of the form $x=t^2, y=at^2+bt^3+$higher order terms for some nonzero $b$ \cite[Theorem 2.1.1]{wall_2004}. It blows up to $x=t^2, y=a+bt+$higher order terms which intersects the exceptional curve $x=0$ at $(0,a)$, corresponding to its tangent. Translate the point $(0,a)$ to the origin, the blow-up parametrizes to $x=Uy^2$ for some unit $U \in \mathbb{C}\llbracket t \rrbracket$, which is smooth at the origin. Thus, there are no more singularities after one blow up.

A bouquet of cusps with distinct tangents intersects the exceptional curve at different coordinates, as the above paragraph shows. Multiple branches do not produce more singularities after one blow up, so the point $P$ has no infinitely near singular point.
\end{proof}



%% file: k_ellipse_journal_sections/nodal_singularities.tex
\section{Affine Singularities}
In this section, we characterize all singularities that are not at infinity. Recall $p_k$ (\ref{pk}) denotes a homogeneous polynomial that defines an algebraic $k$-ellipse.
\begin{dfn}	
A $k$-ellipse with zero radius is called degenerate.
\end{dfn}	
\begin{prop}\label{prop:degdeg}
The irreducible polynomial defining a degenerate $k$-ellipse is $p_k^{1/2}$.
\end{prop}
\begin{proof}

When $r = 0$, $f_{\sigma} = -f_{-\sigma}$ for any $\sigma \in \{\pm 1\}^k$. It can be proved that $p^{1/2}_k$ is a polynomial by induction on $k$. The polynomial $p^{1/2}_k$ vanishes on the $k$-ellipse. By the same arguments as in the proof of \cite[Lemma 2.1]{nie2008}, $p^{1/2}_k$ is irreducible.
\end{proof}

\begin{cor}\label{cor:degdeg}
Let $d_k', m_k'$ denotes the degree and multiplicity at $[\pm \mathrm{i}: 1: 0]$ of a generic degenerate $k$-ellipse, then $d_k' = \frac{d_k}{2}$ and $m_k' = \frac{m_k}{2}$.
\end{cor}

We will show that the affine singularities of a $k$-ellipse are characterized by some corresponding degenerate $j$-ellipses for $0 \leq j < k$. The following example provides some intuition for the degeneracy.

\begin{eg}
A degenerate 2-ellipse looks like a bisecting normal of the two foci. Indeed, $\sqrt{x^2 + y^2} = \sqrt{(x-1)^2 + (y-2)^2}$ contains all points whose distance to (0,0) equals to whose distance to (1,2), defined by $2x + 4y = 5$.
\end{eg}

\begin{thm}\label{thm:affinesing}
Consider an algebraic $k$-ellipse with foci $F = \{(u_i, v_i) \}_{i=1}^k$ and radius $r>0$. For any ordered partition of $F$ into disjoint $F_j \bigsqcup F_{k-j}$ ($2 \leq j \leq k-1$), consider the degenerate $j$-ellipse with foci $F_j$ and the $(k-j)$-ellipse with foci $F_{k-j}$ and radius $r$. Their intersections in the affine plane $z \neq 0$ are singularities of the $k$-ellipse with foci $F$. Conversely, all affine singularities of the $k$-ellipse arise in this way.
\end{thm}
\begin{proof}
Recall $f_\sigma = rz - \sum_{i=1}^k \sigma_i \sqrt{(x-u_i z)^2 + (y-v_i z)^2}$. For $\sigma \in \{1,-1\}^k$ let $\sigma' = (-\sigma_1, \dots, -\sigma_j, \sigma_{j+1}, \dots, \sigma_k)$. At an intersection of a degenerate $j$-ellipses with foci $(u_1, v_1),\dots, (u_j, v_j)$ and a $(k-j)$-ellipse with the rest of the foci and radius $r$, both $f_\sigma$ and $f_{\sigma'}$ vanish.

Three generic algebraic curves do not share a common point. Since $j \in [2, k-1]$ and $r>0$, we may assume that no (reducible) conics $(x - u_i z)^2 + (y - v_i z)^2$ vanish at an intersection of a degenerate $j$-ellipse and a $(k-j)$-ellipse with radius $r$ so the intersection is not a pole of any partial derivatives of any factor. By the Leibniz rule, $q_k$ has all partial derivatives vanish at the intersection.

It remains to show that a $k$-ellipse has no more singularities. At a point $p$ on the $k$-ellipse, at least one of the factors $f_\sigma$ at the right hand side of \eqref{galois} vanishes. If only one factor vanishes, then $p$ does not lie on any circle $(x-u_i z)^2 + (y-v_i z)^2$ or else at least two factors vanish so $p$ is not a pole of any partial derivatives of any factor.

First, we show that a generic configuration of the foci and radius enables $f_\sigma$ not to have singularities that is not an intersection of two analytic branches. Euler's relation \cite[Lemma 2.32]{kirwan_1992} says that if $f(x, y, z)$ is a homogeneous polynomial of degree $m$ then
\[
x \frac{\partial f}{\partial x}(x, y, z) + y \frac{\partial f}{\partial y}(x, y, z) + z \frac{\partial f}{\partial z}(x, y, z) = m f(x, y, z).
\]
By Euler's relation, for a point to be singular on $f_\sigma$, it suffices to consider the common zeros of three out of four equations from $\frac{\partial f_\sigma}{\partial x}, \frac{\partial f_\sigma}{\partial y}, \frac{\partial f_\sigma}{\partial z}$, and $f_\sigma$. We consider $\frac{\partial f_\sigma}{\partial x}, \frac{\partial f_\sigma}{\partial y}$, and $f_\sigma$. Let 
\begin{align*}
d_i = \sqrt{(x-u_i z)^2+(y-v_i z)^2}, \quad
d = \prod_i d_i.
\end{align*}
Observe that for a fixed configurations of foci, 
\begin{align*}
\frac{\partial f_\sigma}{\partial x} = - \sum_i \sigma_i \frac{x-u_iz}{d_i},  \quad
\frac{\partial f_\sigma}{\partial y} = -\sum_i \sigma_i \frac{y-v_iz}{d_i}
\end{align*}
do not depend on $r$. Also, $d \frac{\partial f_\sigma}{\partial x} = - \sum_i \sigma_i (x-u_iz) \prod_{j\neq i}d_i$ is integral over $K$ as $\prod_{j\neq i}d_i$ satisfies the monic equation $t^2 - \prod_{j \neq i} d_i^2 \in K[t]$ where $d_i^2 \in K$. Similarly, $d \frac{\partial f_\sigma}{\partial y}$ is also integral over $K$. Let 
\begin{align*}
f_x = \prod_\sigma d \frac{\partial f_\sigma}{\partial x} = d^{2^k} \prod_\sigma \sum_i \sigma_i \frac{x-u_iz}{d_i},\ f_y = \prod_\sigma d \frac{\partial f_\sigma}{\partial y} = d^{2^k} \prod_\sigma \sum_i \sigma_i \frac{y-v_iz}{d_i},
\end{align*}
both of which are polynomials. Due to the automorphism $\sigma \mapsto -\sigma$, $f_x$ and $f_y$ are both squares up to sign. Let $\tilde{f_x}, \tilde{f_y}$ be their square roots. We may assume  generically that $\tilde{f_x}, \tilde{f_y}$ are irreducible since no proper subproduct of right hand side lies in the ground field $K$. Thus, by B\'ezout's theorem, $V(\tilde{f_x}, \tilde{f_y})$ contains finitely many points. For generic $r$, $f_\sigma$ does not pass through any point in $V(\tilde{f_x}, \tilde{f_y})$, a contradiction. Thus at least two factors vanish at a singular point.

Suppose $p$ is a singular point of a $k$-ellipse. Pick $\sigma$ and $\tau$ such that $f_\sigma, f_\tau$ both vanish at $p$. Let $j = \sum |\sigma_i - \tau_i|$ be the number of entries the two vectors differ. 

\begin{itemize}
\item If $j=1$, we may assume that $k$-th digit differs, then $p$ is on the $(k-1)$-ellipse with the first $(k-1)$ foci. From the assumption that $f_\sigma(p) = f_\tau(p) = 0$ and 
\[
f_\sigma f_\tau = \left(rz - \sum_{i=1}^{k-1} \sigma_i \sqrt{(x-u_iz)^2 + (y-v_iz)^2} \right)^2 - (x-u_k z)^2 - (y-v_k z)^2
\]
we have
\[
\frac{\partial p_k}{\partial z}(p) = \frac{p_k}{f_\sigma f_\tau}(p) \frac{\partial f_\sigma f_\tau}{\partial z}(p) = \frac{p_k}{f_\sigma f_\tau}(p)(u_k^2 + v_k^2) = 0 \iff u_k = \pm i v_k.
\]
With the assumption that $p$ is singular, $\frac{\partial p_k}{\partial x}(p)=0$ and $p = [\pm \mathrm{i}: 1: 0]$. 

\item If $j=k, r>0 \implies z=0$ and again $p = [\pm \mathrm{i}: 1: 0]$.
\end{itemize}

Thus $j \in [2, k-1]$ as was to be shown.
\end{proof}

\begin{figure}
    \centering
    \includegraphics[scale=.4]{movie/30}
    \includegraphics[scale=.4]{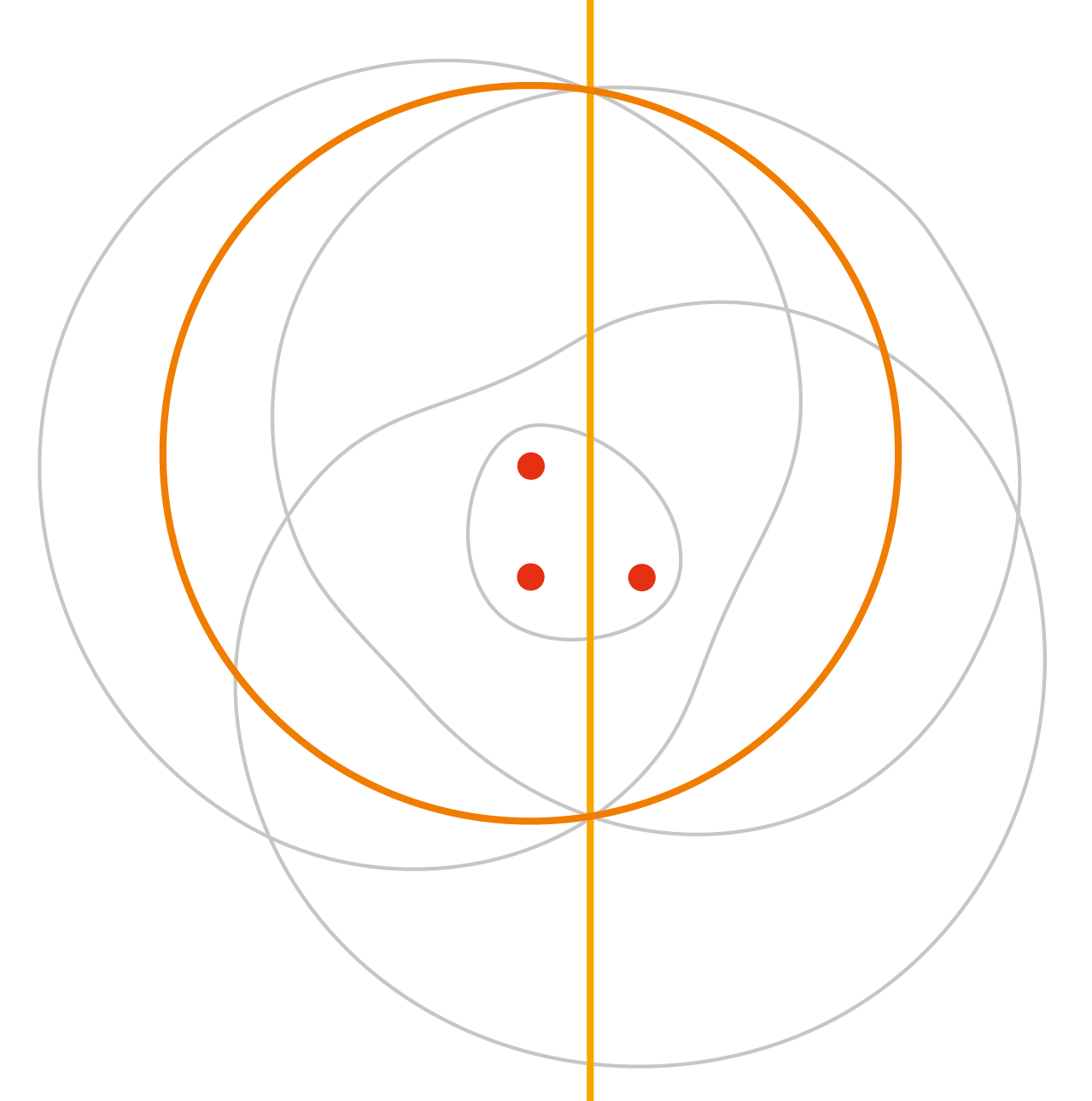}
    \caption{3-ellipse with foci (0,0),(0,1),(1,0) (in $U_z$) and radius 3; bisecting normal of (0,0) and (1,0) intersects circle centered at (0,1) of radius 3 at two nodes of the 3-ellipse.}
    \label{fig:movie30}
\end{figure}

\begin{figure}
    \centering
    \includegraphics[scale=.3]{movie/400}
    \includegraphics[scale=.3]{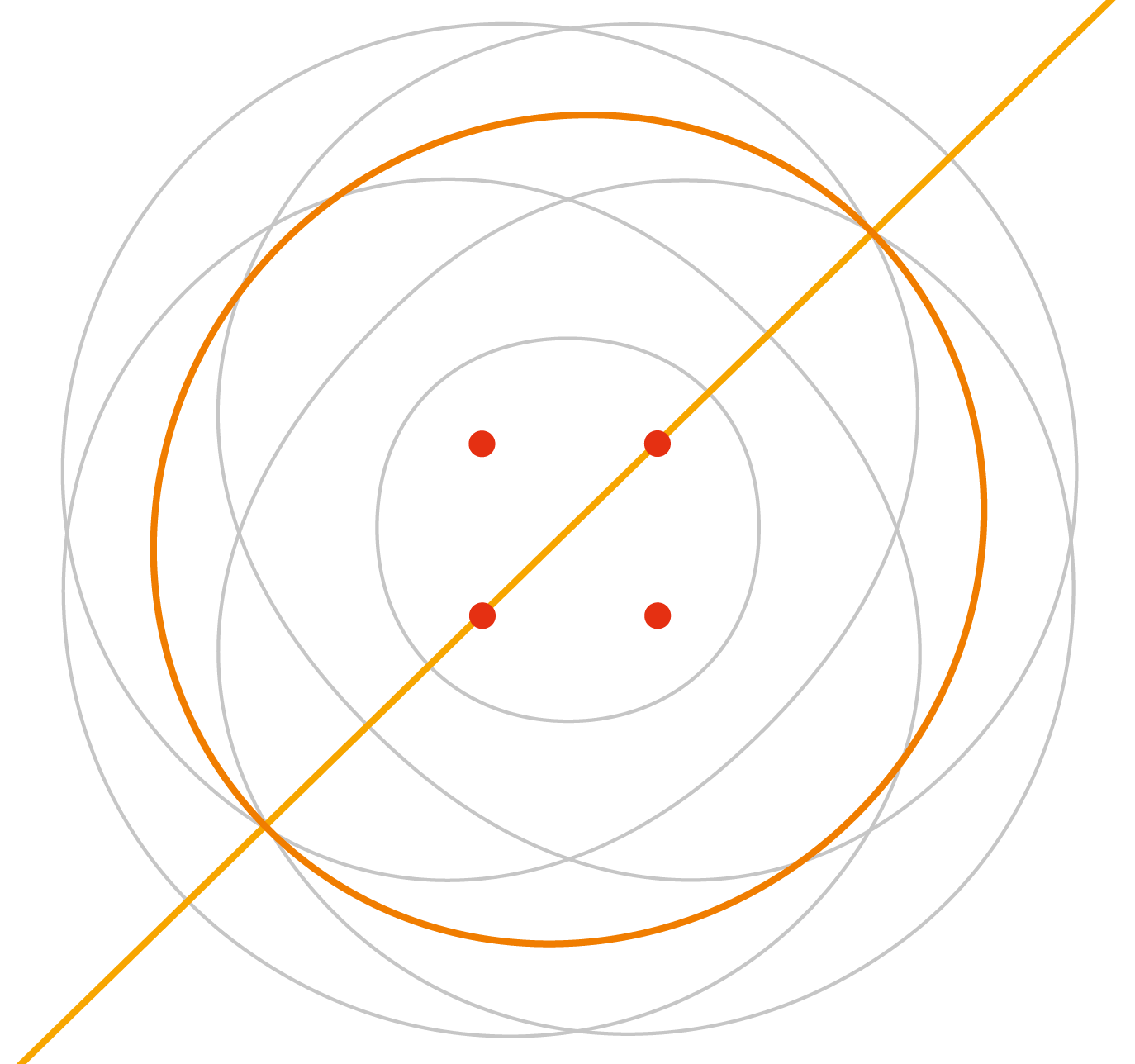}
    \includegraphics[scale=.2]{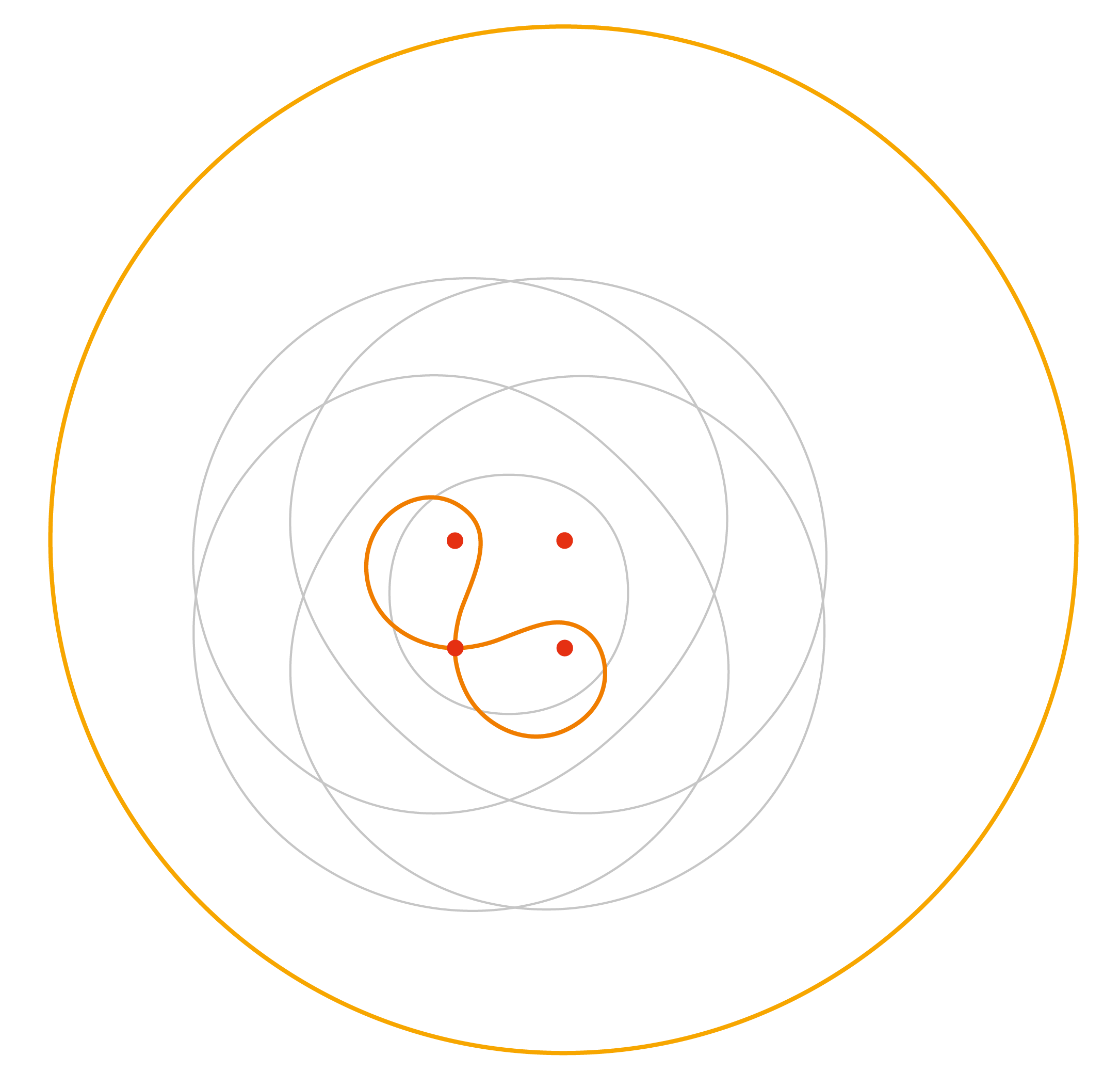}
    \caption{4-ellipse with foci (0,0), (0,1), (1,0), (1,1) (in $U_z$) and radius 5; 2-ellipse with foci (0,0), (1,1) and radius 5 intersects bisecting normal of (1,0) and (0,1) at two nodes of the 4-ellipse; degenerate 3-ellipse with foci (0,0), (0,1), (1,0), and circle centered at (1,1) of radius 5. All their intersections are complex.}
    \label{fig:movie40}
\end{figure}

\begin{figure}
    \centering
    \includegraphics[scale=.3]{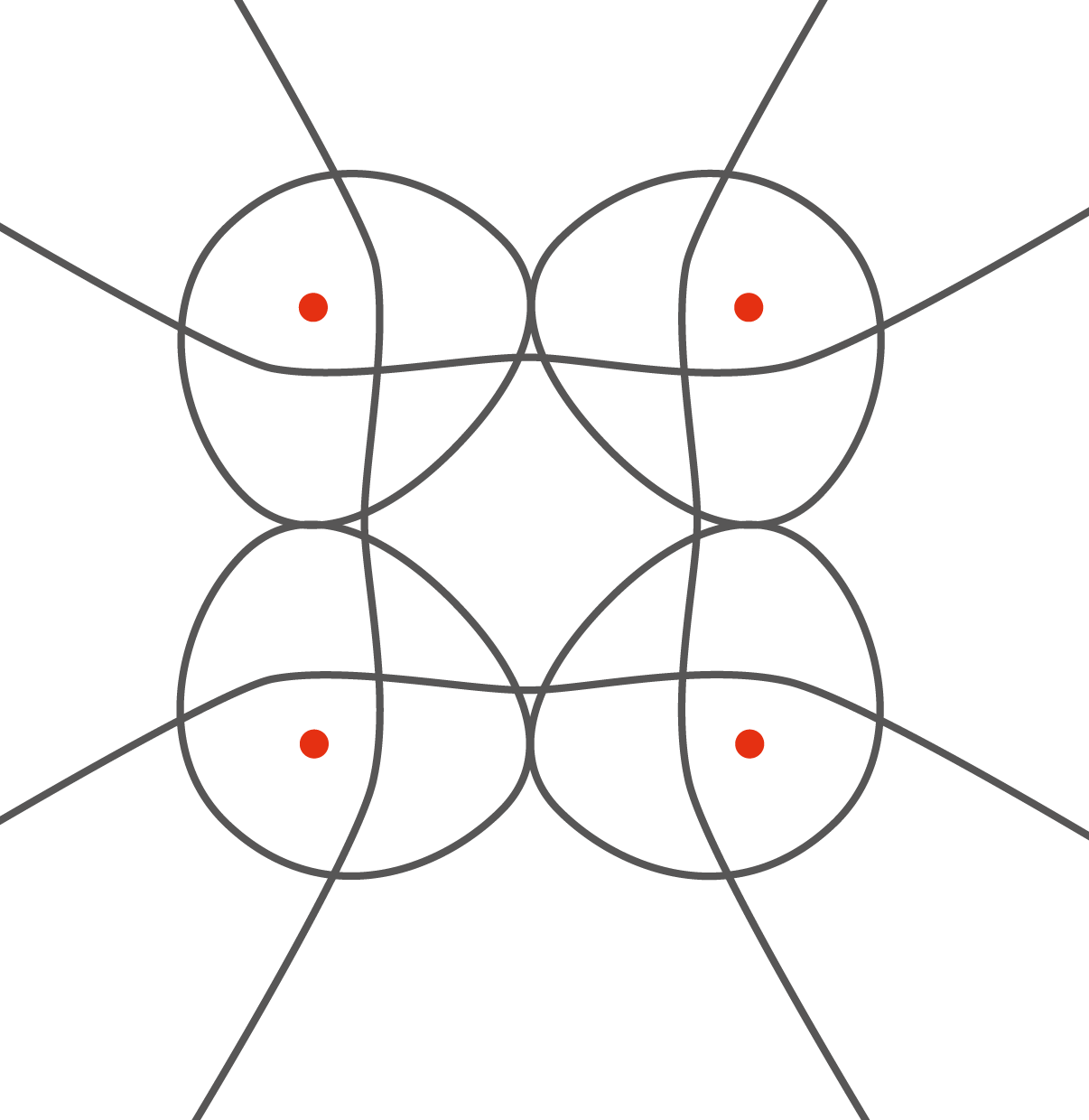}
    \includegraphics[scale=.3]{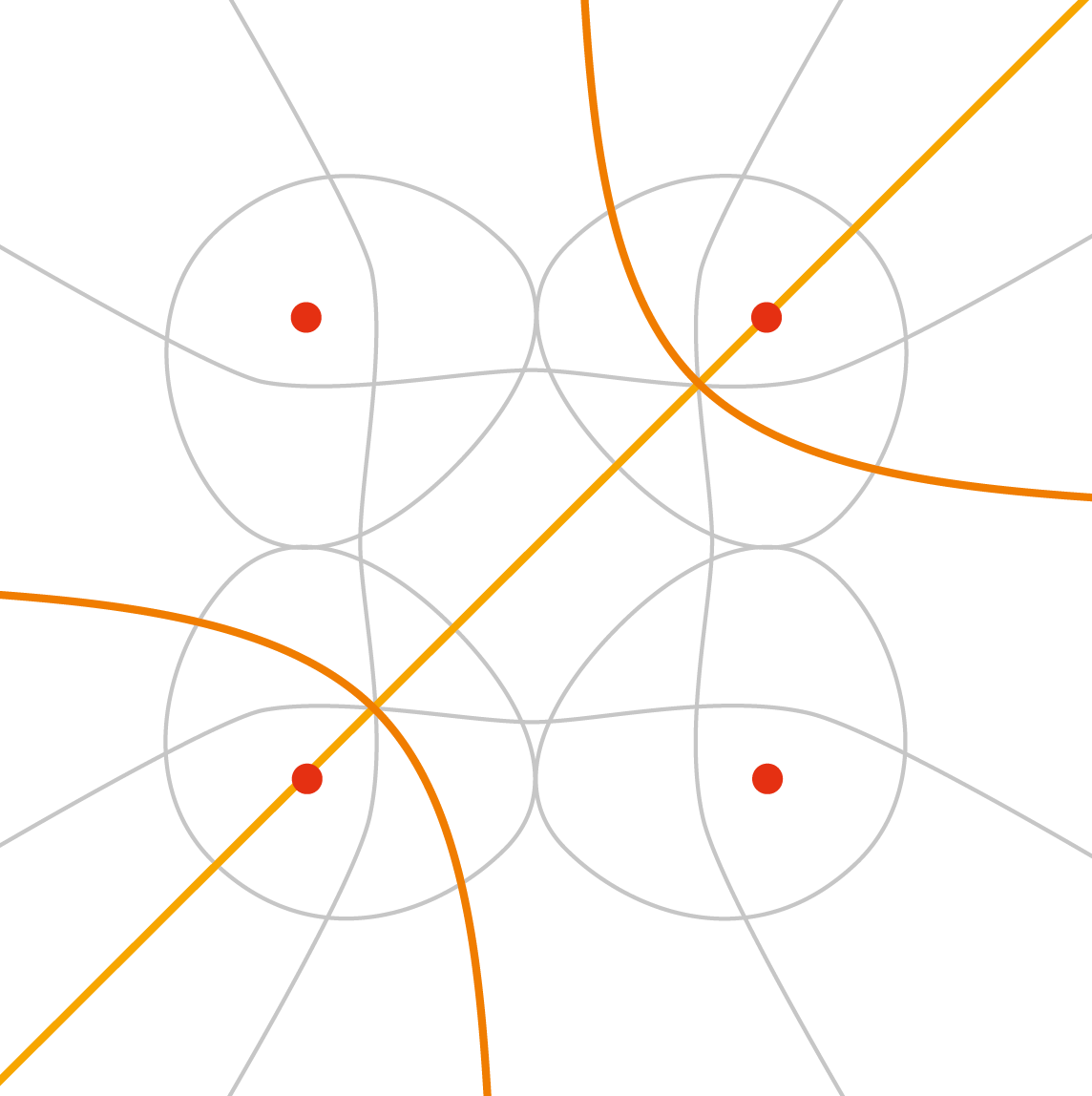}
    \includegraphics[scale=.3]{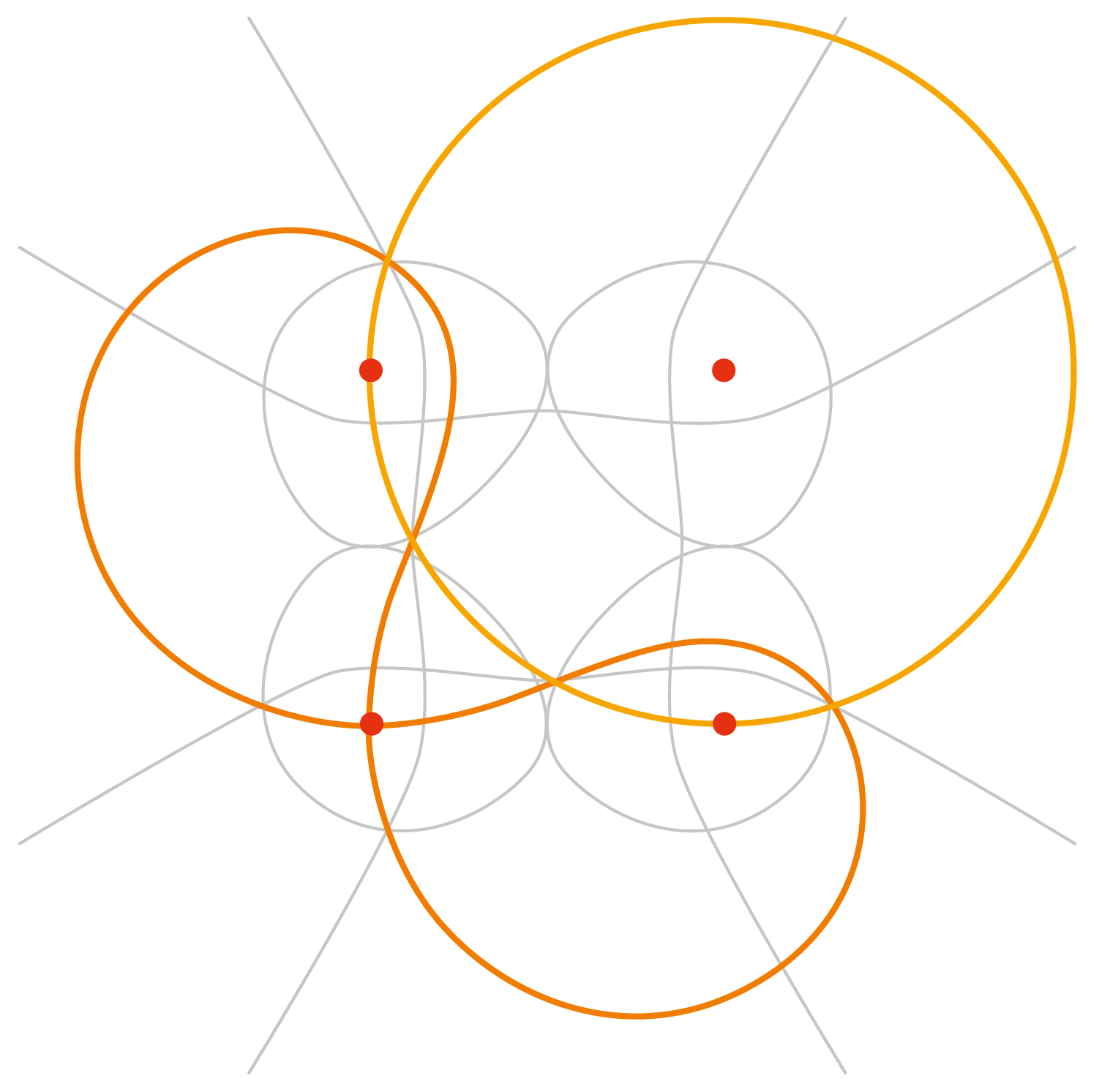}
    \caption{4-ellipse with foci (0,0), (0,1), (1,0), (1,1) (in $U_z$) and radius 1; 2-ellipse(hyperbola) with foci (0,0), (1,1) and radius 1 intersects bisecting normal of (1,0) and (0,1) at two nodes of the 4-ellipse; degenerate 3-ellipse with foci (0,0), (0,1), (1,0) intersects circle centered at (1,1) of radius 1 at four affine real singularities of the 4-ellipse.}
    \label{fig:movie40}
\end{figure}

\begin{eg}
A generic 5-ellipse has 200 nodal singularities. The bisecting normal of a pair of foci intersects a 3-ellipse at $2^3 = 8$ points; a degenerate 3-ellipse intersects a 2-ellipse at $2^2 \times 2 = 8$ points; and a degenerate 4-ellipse intersects a circle at $\frac12(2^4-\binom{4}{2}) \times 2 = 10$ points, where $2 \times 1 \times 1 = 2$ contributed by $[\pm i: 1: 0]$ shall be reduced. There are $\binom{5}{3} = 10, \binom{5}{2} = 10, \binom{5}{1} = 5$ ways to choose foci out of 5 as discussed above, and $200 = 8 \times 10 + 8 \times 10 + 5 \times (10 - 2)$.
\end{eg}

\begin{prop}\label{prop:nodal_singularities}
Except for $[\pm \mathrm{i}: 1: 0]$, a generic $k$-ellipse has 
\[
\sum_{j=2}^{k-1} \binom{k}{j} (d_j' d_{k-j} - 2m_j' m_{k-j}) = \left\{ 
\begin{matrix}
2^{2k-2} - (k+2)2^{k-2} & k\text{ is odd} \\
2^{k-2} - (k+2) 2^{k-2} - \binom{k}{k/2} \left(\binom{k-1}{k/2} - 1 \right) & k \text{ is even} 
\end{matrix} \right.
\]
many nodal singularities.
\end{prop}
\begin{proof}
By B\'ezout's theorem, a generic degenerate $j$-ellipse intersects a $(k-j)$-ellipse at $d_j' d_{k-j}$ points. The summation formula follows from Theorem \ref{thm:affinesing}.

By \cite[Lemma 4.4.2]{wall_2004}, the intersection number of a generic degenerate $j$-ellipse $C$ and a $(k-j)$-ellipse $D$ at $P = [\pm \mathrm{i}: 1: 0]$ is $(C.D)_P = \sum m'_P m_P$ sum over infinitely near singular points they share. Generically we have supposed that a $k$-ellipse has distinct tangent directions at $P$. Then $C$ and $D$ do not share tangent directions at infinity so they do not share infinitely near singular points except for $P$.

When $k$ is odd, $j$ and $k-j$ has different parity. Without loss of generality, $j$ is even, then $d_j' d_{k-j} - 2m_j' m_{k-j} = d_j d'_{k-j} - 2m_j m_{k-j}' = 2^{k-2}$. Then by binomial formula,
\[
\sum_{j=2}^{k-1} \binom{k}{j} (d_j' d_{k-j} - 2m_j' m_{k-j}) = 2^{k-2} (2^k - 2 - k) = 2^{2k-2} - (k+2) 2^{k-2}.
\]

When $k$ is even, $j$ and $k-j$ have the same parity. If $j$ is odd, $d_j' d_{k-j} - 2m_j' m_{k-j} = 2^{k-2}$; otherwise, $d_j' d_{k-j} - 2m_j' m_{k-j} = 2^{k-2} - {1\over 2} \binom{j}{j/2} \binom{k-j}{\frac{k-j}{2}}$. 
\[
\sum_{j=2}^{k-1} \binom{k}{j} (d_j' d_{k-j} - 2m_j' m_{k-j}) = \sum_{j=2}^{k-1} \binom{k}{j}2^{k-2} - {1\over 2}\sum_{j=2, \text{even}}^{k-1} \binom{k}{j} \binom{j}{j/2} \binom{k-j}{\frac{k-j}{2}} 
\]
On the right hand side, the first term $\sum_{j=2}^{k-1} \binom{k}{j}2^{k-2} = 2^{2k-2} - (k+2) 2^{k-2}$. By the identity $\sum_{i=0}^n \binom{n}{i} \binom{n}{n-i} = \binom{2n}{n}$ for all $n \in \mathbb{N}$ and $\binom{k-1}{k/2} = {1\over 2}\binom{k}{k/2}$ for $k$ even, the second term is
\begin{align*}
& \quad {1\over 2}\sum_{j=2, \text{even}}^{k-1} \binom{k}{j} \binom{j}{j/2} \binom{k-j}{\frac{k-j}{2}} 
= \binom{k-1}{k/2} \sum_{j=2, \text{even}}^{k-1} \binom{k/2}{j/2} \binom{k/2}{\frac{k-j}{2}} \\
&= \binom{k-1}{k/2} \left( \binom{k}{k/2} - 2 \right) = \binom{k}{k/2} \left(\binom{k-1}{k/2} - 1 \right).
\end{align*}
\end{proof}
\begin{proof}[Proof of Theorem \ref{thm:genus}]
By Theorem \ref{thm:genus_formula}
\begin{align}
    g = \binom{d-1}{2} - \sum_{P \in Sing (C)} \delta_P. \nonumber
\end{align}

$Sing (C)$ consists of affine nodal singularities, whose number is counted in Proposition \ref{prop:nodal_singularities}, and singularities at infinity $P = [\pm \mathrm{i}:1:0]$.
Note that $\delta_{\text{node}} = 1$. 
By Proposition \ref{prop:sing_at_inf_no_infly_near_sing} and Theorem \ref{thm:delta_and_multiplicity}, for $P = [\pm \mathrm{i}: 1: 0], \delta_P = {m_P \choose 2}$, and by Proposition \ref{prop:tangentcone}, 
\[
m_P = \begin{cases} 2^{k-1} &  k \text{ odd}\\
2^{k-1} - {k \choose k/2} &  k \text{ even}\\
\end{cases}
\]
We are now ready compute the genus $g$.

When $k$ is odd, $g = \frac{(2^k - 1) (2^k - 2)}{2} - \sum_{P \in Sing (C)} \delta_P = (2^{2k -1} - 3 \cdot 2^{k-1} + 1) - (2^{2k-2} - (k+2)2^{k-2}) - 2^{k-1}(2^{k-1} - 1) = (k-2)2^{k-2} + 1$.

When $k$ is even, 
\begin{align*}
g &= \frac{1}{2} \left(2^k - \binom{k}{k/2}-1 \right) \left(2^k - \binom{k}{k/2}-2 \right) - \sum_{P \in Sing (C)} \delta_P \\
&= (k-2) 2^{k-2} + 1 - \binom{k-1}{k/2}.
\end{align*}
\end{proof}

\begin{proof}[Proof of Theorem \ref{thm:singularities}]
It follows from Theorem \ref{thm:affinesing} and Proposition \ref{prop:nodal_singularities}.
\end{proof}

%% file: k_ellipse_journal_sections/dual_curve.tex
\section{Dual curve}
In this last section, we give the semidefinite representation of the dual curve of the $k$-ellipse and prove the theorem on the degree of the dual curve of the algebraic $k$-ellipse.

Let $L_k(x, y) = x\cdot A_k + y \cdot B_k + C_k$ be as defined in Section \ref{section:preliminaries}.
By \cite{nie2008}, the convex set bounded by the $k$-ellipse in $\R^2$ is defined by the matrix inequality $\mathcal{E}_k = \{(x, y) \in\R^2: L_k(x, y) \succeq 0\}$, and the $k$-ellipse is the set of all solutions to the semidefinite programming (SDP) problem
\begin{align*}
    & \text{minimize}_{x, y} \ \alpha x + \beta y\\
    & \text{subject to} \ L_k(x, y) \succeq 0
\end{align*}
where $\alpha, \beta$ run over $\R$.
We want to get a similar semidefinite representation for the dual curve.
To this end, we consider the ``dual set" of the convex set bounded by the $k$-ellipse:
\begin{dfn}
The \textit{(geometric) polar} of a convex set $S$ is defined as $S^\circ = \{y|y^\top x \leq 1, \forall x \in S\}$.
\end{dfn}
The polar of $\mathcal{E}_k$, $\mathcal{E}^\circ_k = \{(w_1,w_2)\in \R^2|w_1 x + w_2 y \leq 1, x\cdot A_k + y \cdot B_k + C_k\succeq 0\}$, is the convex set bounded by the dual curve of the $k$-ellipse.
To get the explicit representation for the polar $\mathcal{E}^\circ_k$, i.e., to get rid of the intermediate variables $x$ and $y$, we define:
\begin{dfn}[\cite{ramana1995some}]
The \textit{algebraic polar}  for the set $G = \{x | Q(x) \succeq 0\}$, where $Q(x) = Q_0 + \sum_i x_i Q_i, Q_i \in \mathbb{S}_n, \forall i$, is $G^* = \{-L(X) | X\cdot Q_0 \leq 1, X \succeq 0\}$, where $L(X)_i = X\cdot Q_i, \forall i$.
Here the operator $\cdot$ between two matrices $A$ and $B$ denotes the inner product in $\mathbb{S}^n$, $A\cdot B := \tr(A^\top B)$.
\end{dfn}
So by definition, the algebraic polar of $\mathcal{E}_k$ is $\mathcal{E}^*_k = \{(-X\cdot A_k, - X\cdot B_k) | X \cdot C_k \leq 1, X\succeq 0\}$.
By \cite[Lemma 4]{ramana1995some}, if $\mathbf{0} \in \mathcal{E}_k$, i.e., $C_k\succeq 0$, then $\mathcal{E}^\circ_k = Cl(\mathcal{E}^*_k)$, where $Cl(\cdot)$ denotes the closure.
The dual curve $(-X\cdot A_k, - X\cdot B_k)$ can therefore be computed by the SDP
\begin{align*}
    &\text{minimize}_X \ W\cdot X\\
    &\text{subject to} \ X \cdot C_k = 1, X\succeq 0
\end{align*}
where $W$ runs over all symmetric matrices.

\begin{figure}
    \centering
\includegraphics[scale=.5]{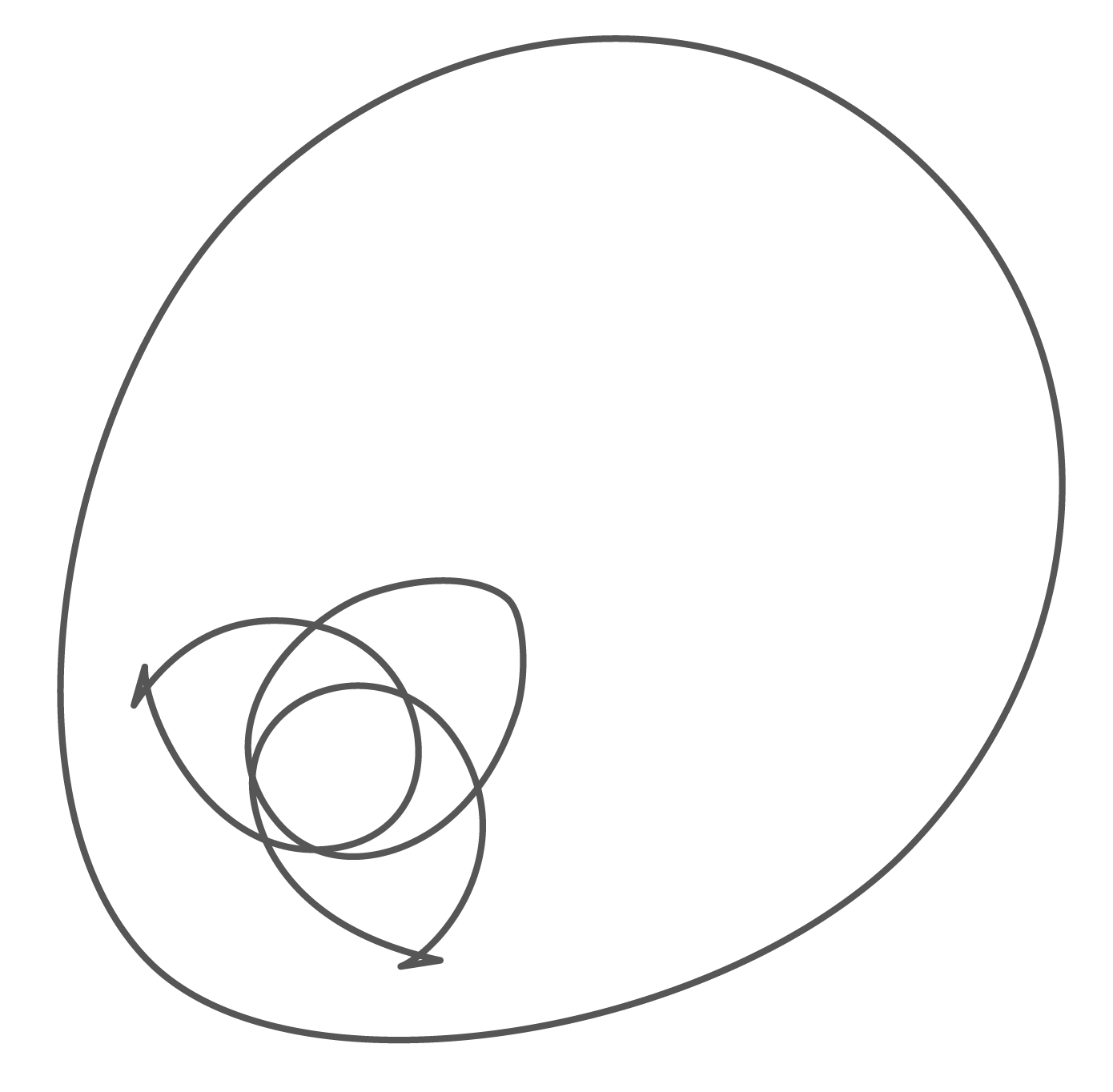}
\caption{Dual of the 3-ellipse with foci (0,0), (1,0), (0,1) (in $U_z$) and radius 3.}
    \label{fig:my_label}
\end{figure}

Finally we finish the proof of the theorem on the degree of the dual curve.
\begin{proof}[Proof of Theorem \ref{thm:dual}]
By \cite[Corollary 7.2.3]{wall_2004}, the degree of the dual curve satisfies $d^\vee - 2g = 2(d-1) - \sum_{P \in Sing (C)} (m_P - r_P)$.
Note that $m_P = r_P$ at a node, and by Corollary \ref{cor:degdeg}, $m_P = 2r_P$ at $P = [\pm \mathrm{i}: 1: 0]$.
When $k$ is odd, $d^\vee = 2(g+d-1) - \sum_{P \in Sing (C)} (m_P - r_P) = 2((k-2) 2^{k-2} + 1 + 2^k - 1) - 2^{k-1} = (k+1)2^{k-1}$.
When $k$ is even,
\begin{align*}
d^\vee &= 2(g+d-1) - \sum_{P \in Sing (C)} (m_P - r_P) \\
&= 2 \left((k-2) 2^{k-2} - \binom{k-1}{k/2} + 1 + 2^k - \binom{k}{k/2} - 1 \right) - \left(2^{k-1} - \binom{k}{k/2} \right) \\
&= (k+1)2^{k-1} - 2\binom{k}{k/2}.
\end{align*}
\end{proof}

%% file: k_ellipse_journal_sections/conclusion.tex
\section{Conclusion and Future Work}
We have characterized the singularities of the algebraic $k$-ellipse, and proved the formula for its genus. We have given the LMI formulation of the dual curve of the $k$-ellipse and proved the formula for the degree of the dual curve of the algebraic $k$-ellipse.

Knowledge about singular locus enables us to compute the adjoint series and hence the canonical model of the $k$-ellipses. We are interested in the canonical image of the ellipse family in the moduli space of corresponding genus. In particular, the moduli space of genus 3 curves has dimension $3g-3 = 6$, and the 3-ellipses have 6 degrees of freedom up to projective transformation. Whether every smooth quartic is the canonical model of some 3-ellipse is a possible direction for future research.

%% file: k_ellipse_journal_sections/acknowledgement.tex
\section{Acknowledgement}
We are really grateful to Bernd Sturmfels for his helpful advice throughout the writing of the paper. We also thank Madeline Brandt, Turku Ozlum Celik, Avinash Kulkarni, Pablo Parrilo, Daniel Plaumann, Qingchun Ren, Shamil Shakirov, and Rainer Sinn for helpful discussions at various stages of the writing of the paper.